\documentclass[a4paper,10pt]{amsart}
\usepackage[UKenglish]{babel}
\usepackage{amsmath,amssymb,amsthm}
\usepackage[pdftex]{color}
\usepackage[bookmarks=true,hyperindex,pdftex,colorlinks, citecolor=MidnightBlue,linkcolor=MidnightBlue, urlcolor=MidnightBlue]{hyperref}
\usepackage[dvipsnames]{xcolor}

\usepackage{enumerate}

\newtheorem{teo}{Theorem}[section]
\newtheorem{prop}[teo]{Proposition}
\newtheorem{cor}[teo]{Corollary}
\newtheorem{lem}[teo]{Lemma}
\newtheorem{ejem}[teo]{Example}

\theoremstyle{definition}
\newtheorem{defi}[teo]{Definition}

\theoremstyle{remark}
\newtheorem{remark}[teo]{Remark}

\renewcommand{\geq}{\geqslant}
\renewcommand{\leq}{\leqslant}

\newcommand{\N}{\mathbb{N}}
\newcommand{\R}{\mathbb{R}}

\newcommand{\SA}{\operatorname{LipSNA}}
\newcommand{\NA}{\operatorname{NA}}
\newcommand{\Lip}{{\mathrm{Lip}}_0}
\newcommand{\Lipc}{{\mathrm{Lip}}_{0K}}
\newcommand{\diam}{\operatorname{diam}}

\newcommand{\cco}{\overline{\operatorname{co}}}
\newcommand{\co}{\operatorname{co}}
\newcommand{\Mol}{\operatorname{Mol}}
\newcommand{\F}{\mathcal{F}}

\begin{document}

\title[The Bishop--Phelps--Bollob\'{a}s property for Lipschitz maps]{The Bishop--Phelps--Bollob\'{a}s property\\ for Lipschitz maps}

\author[Chiclana]{Rafael Chiclana}

\author[Mart\'in]{Miguel Mart\'in}

\address{Universidad de Granada, Facultad de Ciencias.
Departamento de An\'{a}lisis Matem\'{a}tico, 18071-Granada
(Spain)}
\email{rafachiclanavega@gmail.com; mmartins@ugr.es}

\thanks{Research partially supported by Spanish MINECO/FEDER grants MTM2015-65020-P and PGC2018-093794-B-I00}

\keywords{Lipschitz function; Lipschitz map; Lipschitz-free space; norm attaining operators; Bishop-Phelps-Bollobas property; metric space}

\subjclass[2010]{Primary 46B04; Secondary 46B20, 26A16, 54E50}

\date{January 24th, 2019. Revised: May 2nd, 2019.}

\begin{abstract}
In this paper, we introduce and study a Lipschitz version of the Bishop-Phelps-Bollob\'{a}s property (Lip-BPB property). This property deals with the possibility of making a uniformly simultaneous approximation of a Lipschitz map $F$ and a pair of points at which $F$ almost attains its norm by a Lipschitz map $G$ and a pair of points such that $G$ strongly attains its norm at the new pair of points. We first show that if $M$ is a finite pointed metric space and $Y$ is a finite-dimensional Banach space, then the pair $(M,Y)$ has the Lip-BPB property, and that both finiteness assumptions are needed. Next, we show that if $M$ is a uniformly Gromov concave pointed metric space (i.e.\ the molecules of $M$ form a set of uniformly strongly exposed points), then $(M,Y)$ has the Lip-BPB property for every Banach space $Y$. We further prove that this is the case for finite concave metric spaces, ultrametric spaces, and H\"{o}lder metric spaces. The extension of the Lip-BPB property from $(M,\R)$ to some Banach spaces $Y$ and some results for compact Lipschitz maps are also discussed.
\end{abstract}

\maketitle

\thispagestyle{plain}

\section{Introduction}
	
Throughout this paper the Banach spaces will be over the real scalars and all the metric spaces will be complete. If $X$ is a Banach space, we will denote the unit ball of $X$ by $B_{X}$, the unit sphere by $S_{X}$, and the topological dual of $X$ by $X^*$. We denote by $\operatorname{L}(X,Y)$ the space of bounded linear operators between the Banach spaces $X$ and $Y$.

Given a \emph{pointed} metric space $M$ (that is, there is a distinguished point in $M$ denoted by $0$) and a Banach space $Y$, the space $\Lip(M,Y)$ of all Lipschitz maps from $M$ to $Y$ which vanish at $0$ is a Banach space when endowed with the norm
	\[ \|F\|_L = \sup \left \{\frac{\|F(p)-F(q)\|}{d(p,q)} \colon p, q \in M,\, p\neq q \right \} \quad \forall \, F \in \Lip(M,Y).\]
	We say that $F \in \Lip(M,Y)$ \emph{attains its norm in the strong sense} or \emph{strongly attains its norm} if there exist $p, q \in M$, $p\neq q$ such that
	\[ \frac{\|F(p)-F(q)\|}{d(p,q)}=\|F\|_L. \]
We write $\SA(M,Y)$ to denote the set of those Lipschitz maps from $M$ to $Y$ which strongly attain their norms.

Recently, the problem of deciding for which metric spaces $M$ the set $\SA(M,Y)$ is (norm) dense in $\Lip(M,Y)$ has been studied. We refer the reader to \cite{articulo1}, \cite{lppr}, \cite{Godefroy-survey-2015}, and \cite{kms}, as references for this study. Let us collect here some of the known results. First, the density does not always hold, as in \cite[Example 2.1]{kms} it is shown that $\SA([0,1],\mathbb{R})$ is not dense in $\Lip([0,1],\mathbb{R})$. In fact, this can be generalized to length spaces \cite[Theorem 2.2]{articulo1}. On the other hand, it is obvious that $\SA(M,Y)=\Lip(M,Y)$ when $M$ is finite (actually, this fact characterizes finiteness of $M$). Besides, $\SA(M,Y)$ is dense in $\Lip(M,Y)$ for every Banach space $Y$ when $M$ is uniformly discrete, or $M$ is countable and compact, or $M$ is a compact \emph{H\"{o}lder metric space} (i.e.\ $M=(N,d^\theta)$ for some metric space $(N,d)$ and $0<\theta <1$), but to understand these results we need a little background. Let $M$ be a pointed metric space. We consider $\delta\colon M\longrightarrow \Lip(M,\mathbb{R})^*$ the canonical isometric embedding $p\longmapsto \delta_p$ given by
\[
\langle f,\delta_p\rangle=f(p) \quad \text{for every $p \in M$ and $f \in \Lip(M,\mathbb{R})$}.
\]
We denote by $\mathcal{F}(M)$ the norm-closed linear span of $\delta(M)$ in the Banach space $\Lip(M,\mathbb{R})^*$, which is usually called the \emph{Lipschitz-free} space over $M$. For background on this, we refer to the papers \cite{Godefroy-survey-2015}, \cite{gk}, and the book \cite{wea2}. It is well known that $\mathcal{F}(M)$ is an isometric predual of the space $\Lip(M,\mathbb{R})$ (in fact, if $M$ is bounded or geodesic it is the unique isometric predual \cite[Section 3.4]{wea2}). In \cite[Proposition 7.4]{lppr} it is proved that if $\mathcal{F}(M)$ has the Radon-Nikodym property (RNP), then $\SA(M,Y)$ is dense in $\Lip(M,Y)$ for every Banach space $Y$ (see also \cite[Theorem~3.1]{articulo1} for an alternative proof). This result provides the proof of the density of $\SA(M,Y)$ in $\Lip(M,Y)$ for every Banach space $Y$ in the aforementioned examples: when $M$ is uniformly discrete, or $M$ is countable and compact, or $M$ is a compact H\"{o}lder metric space, as $\mathcal{F}(M)$ has the RNP in all these cases (see \cite[Example~1.2]{articulo1} to find references where this was proved for each example). Other sufficient conditions for the denseness are also given in \cite{articulo1}, as the existence of a norming set of uniformly strongly exposed points in the unit ball of $\mathcal{F}(M)$ \cite[Proposition 3.3]{articulo1}, in particular, when $\mathcal{F}(M)$ has property $\alpha$ \cite[Corollary 3.10]{articulo1}.

All the mentioned positive results for the density of strongly norm-attaining Lipschitz maps use the fact that every Lipschitz map $F\colon M \longrightarrow Y$ can be isometrically identified with the bounded linear operator $\widehat{F} \colon \mathcal{F}(M) \longrightarrow Y$ defined by $\widehat{F}(\delta_p)=F(p)$ for every $p \in M$. Actually, this provides a total identification (i.e.\ an isometric isomorphism) from $\Lip(M,Y)$ onto $\operatorname{L}(\mathcal{F}(M),Y)$ which we will profusely use throughout the paper. We write $\Mol(M)$ to denote the set of all elements of $\mathcal{F}(M)$ of the form
	\[ m_{p,q} =\frac{\delta_p-\delta_q}{d(p,q)} \quad  p, q \in M,\, p\neq q, \]
which are called \emph{molecules}. A straightforward application of the Hahn-Banach Theorem implies that
	\[ \overline{\co}(\Mol(M)) = B_{\mathcal{F}(M)}.\]
With this notation, a Lipschitz map $F \colon M \longrightarrow Y$ strongly attains its norm when there exists $m_{p,q} \in \Mol(M)$  such that  $\|F\|_L=\|\widehat{F}(m_{p,q})\|$. Let us observe that if $M$ is actually a Banach space and $F:M\longrightarrow Y$ is a bounded linear operator, then $F$ strongly attains its norm (viewed as a Lipschitz function from $M$ into $Y$) if and only if $F$ attains its norm as a bounded linear operator from $M$ into $Y$. In this way, the problem of deciding for which pointed metric spaces $M$ and Banach spaces $Y$ we will have that $\SA(M,Y)$ is dense in $\Lip(M,Y)$ can be considered both as a non-linear generalization of the study of norm-attaining bounded linear operators and as a particular case of such study when the domain space is a Lipschitz-free space, using a stronger definition of norm-attainment.

The study of norm-attaining operators was initiated by Lindenstrauss \cite{lindens} in the 1960's, trying to extend to operators the Bishop-Phelps theorem \cite{bishop-phelps} which states that the set of functionals which attains their norm on a Banach space $X$ is always dense in $X^*$. We refer to the expository papers \cite{AcostaRACSAM2006}, \cite{martin} for a detailed account on results on norm-attaining operators. Let us write $\NA(X,Y)$ to denote the set of norm-attaining operators from the Banach space $X$ to the Banach space $Y$. Observe that for a pointed metric space $M$ and a Banach space $Y$, the density of $\SA(M,Y)$ in $\Lip(M,Y)$ clearly implies the density of $\NA(\mathcal{F}(M),Y)$ in $\operatorname{L}(\mathcal{F}(M),Y)$, but the reciprocal result is no longer true: $\SA([0,1],\R)$ is not dense in $\Lip([0,1],\R)$, while $\operatorname{L}(\mathcal{F}([0,1]),\R)$ is dense in $\Lip(M,\R)=\mathcal{F}(M)^*$ by the Bishop-Phelps Theorem.

An extension of the Bishop-Phelps Theorem was given by Bollob\'{a}s \cite{bollobas} in 1970, which shows that one is always able to make a simultaneous approximation of a functional $f$ and a vector $x$ at which $f$ almost attains its norm by a functional $g$ and a vector $y$ such that $g$ attains its norm at $y$. To study the validity of this result for operators, a property was introduced in 2008. A pair of Banach spaces $(X,Y)$ has the {Bishop-Phelps-Bollob\'{a}s property} (\emph{BPBp} in short) \cite{BPBp} if given $\varepsilon>0$, there is $\eta(\varepsilon)>0$ such that for every norm-one $T \in \operatorname{L}(X,Y)$ and every $x \in S_X$ such that $\|T(x)\|>1-\eta(\varepsilon)$, there exist $u \in S_X$ and $S \in \operatorname{L}(X,Y)$ satisfying
		\[ \|S(u)\|=\|S\|=1,\quad \|T-S\|< \varepsilon, \quad \|x-u\|<\varepsilon. \]
If an analogous definition is valid for operators $T$ and $S$ belonging to a subspace $\mathcal{M} \subseteq \operatorname{L}(X,Y)$, then we say that $(X,Y)$ has the BPBp for operators from $\mathcal{M}$. There is a vast literature about this topic, and we refer the reader to the already cited \cite{BPBp}, to \cite{acklm,ACK,absolutesums,dgmm}, and to the references therein. Let us comment that the mentioned result by Bollob\'{a}s just says that the pair $(X,\R)$ has the BPBp for every Banach space $X$. It is clear that the BPBp of a pair $(X,Y)$ implies the density of $\NA(X,Y)$ in $\operatorname{L}(X,Y)$. The reciprocal result is far for being true: if $Y$ is a strictly convex Banach space which is not uniformly convex, then the pair $(\ell_1^2,Y)$ fails the BPBp (see \cite[Lemma 3.2]{acklm}), while $\NA(\ell_1^2,Y)=\operatorname{L}(\ell_1^2,Y)$.

Our aim in this paper is to extend the Bishop-Phelps-Bollob\'{a}s property to the Lipschitz context in a natural way. Let $M$ be a pointed metric space and let $Y$ be a Banach space. The role of the norm-attaining operator $S$ will be played by a strongly norm-attaining Lipschitz map or, equivalently, by an element of $\operatorname{L}(\mathcal{F}(M),Y)$ attaining its norm at a molecule. As the set $\Mol(M)$ is closed in norm \cite[Proposition 2.9]{lppr} and so the only elements in the unit sphere of $\mathcal{F}(M)$ that can be approximated by molecules are molecules, we need to restrict the point $x$ to be a molecule. Therefore, our generalization reads as follows.

\begin{defi}\label{Lip-BPBp}
		Let $M$ be a pointed metric space and let $Y$ be a Banach space. We say that the pair $(M,Y)$ has the \emph{Lipschitz Bishop-Phelps-Bollob\'{a}s property} (\emph{Lip-BPB property} for short), if given $\varepsilon>0$ there is $\eta(\varepsilon)>0$ such that for every norm-one $F\in \Lip(M,Y)$ and every $p,q\in M$, $p\neq q$ such that $\|F(p)-F(q)\|> \bigl(1-\eta(\varepsilon)\bigr)d(p,q)$, there exist $G \in \Lip(M,Y)$ and $r,s\in M$, $r\neq s$, such that
\[
\frac{\|G(r)-G(s)\|}{d(r,s)}=\|G\|_L=1,\quad \|G-F\|_L<\varepsilon, \quad \frac{d(p,r)+d(q,s)}{d(p,q)}<\varepsilon.
\]
If the previous definition holds for a class of linear operators from $\mathcal{F}(M)$ to $Y$, we will say that the pair $(M,Y)$ has the Lip-BPB property for that class.
\end{defi}

It is clear that if a pair $(M,Y)$ has the Lip-BPB property, then $\SA(M,Y)$ will be norm-dense in $\Lip(M,Y)$. We will show that the reciprocal result is no longer true (see Example \ref{finito no}, among others).

Observe that the quantity $\frac{d(p,r)+d(q,s)}{d(p,q)}$ in the definition above measures the nearness of the pair $(p,q)$ to the pair $(r,s)$ modulated by the distance of $p$ to $q$, so the smallness of it represents that the two pairs are ``relatively'' near one to the other.

The first statement of the following remark, which is a routine application of Lemma 1.3 of \cite{articulo1} analogous to what is done in \cite[Lemma 3.12]{articulo1}, gives a reformulation of the Lip-BPB property. We will use both equivalent formulations without giving any explicit reference. The second statement describes the case when $M$ is a Banach space and the Lipschitz maps are actually linear.

\begin{remark}
Let $M$ be a pointed metric space and let $Y$ be a Banach space.
\begin{enumerate}
  \item[(a)] The pair $(M,Y)$ has the Lip-BPB property if and only if given $\varepsilon>0$ there is $\eta(\varepsilon)>0$ such that for every norm-one $\widehat F\in \operatorname{L}(\mathcal{F}(M),Y)$ and every $m\in \Mol(M)$ such that $\|\widehat F(m)\|>1-\eta(\varepsilon)$, there exist $\widehat G \in \operatorname{L}(\mathcal{F}(M),Y)$ and $u \in \Mol(M)$ such that
\[
\|\widehat{G}(u)\|=\|G\|_L=1,  \quad \|\widehat F-\widehat G\| <\varepsilon, \quad \|m-u\|<\varepsilon. \]
  \item[(b)] Suppose $M$ is a Banach space. If in  Definition \ref{Lip-BPBp}, for every $F\in \operatorname{L}(M,Y)$ satisfying the hypothesis we actually get $G\in \operatorname{L}(M,Y)$ satisfying the thesis, then we recuperate the (classical) BPBp of the pair $(M,Y)$.
\end{enumerate}
\end{remark}

The above two remarks shows that the study of the LipBPB property is both a non-linear generalization of the (classical) BPBp and a particular case of the BPBp where the domain space is a Lipschitz-free space and the concept of norm-attainment is stronger than the usual one.

Let us present the content of the paper. Section \ref{sec-finite} is devoted to present some results for finite metric spaces. In particular, we show that if $M$ is a finite pointed metric space and $Y$ is a finite-dimensional space, then $(M,Y)$ has the Lip-BPB property. We also present examples showing that this result is no longer true without the finiteness of the metric space or without the finite-dimensionality of the Banach space. We prove in Section \ref{sec-universal} that if $M$ is uniformly Gromov concave (i.e.\ $\Mol(M)$ is a set of uniformly strongly exposed points) then $(M,Y)$ has the Lip-BPB property for every Banach space $Y$. Examples of such $M$'s are ultrametric spaces, concave finite pointed metric spaces, and H\"{o}lder metric spaces. Finally, we divide Section \ref{sec-fromscalartovectors} into two parts. The first part is devoted to discuss the relationship between the Lip-BPB property for scalar functions and the Lip-BPB property for vector-valued maps. We first show that the Lip-BPB property of $(M,\R)$ is a necessary condition to have that $(M,Y)$ has the Lip-BPB property for any space $Y$. Then, we present property $\beta$ as a sufficient condition on $Y$ to assure that the Lip-BPB property of $(M,\R)$ passes to the Lip-BPB property of $(M,Y)$. Moreover, we show that assuming the density of $\SA(M,\mathbb{R})$, property quasi-$\beta$ is a sufficient condition on $Y$ to guarantee the density of $\SA(M,Y)$, and that the same result does not hold for the Lip-BPB property. The second part contains results for Lipschitz compact maps. Among them, we show that the Lip-BPB property of $(M,\R)$ implies the Lip-BPB property for Lipschitz compact maps of $(M,Y)$ when $Y$ is a predual of an $L_1$-space.

Let us finally mention that some of the proofs are inspired by the analogous ones for the linear BPBp.
	
	\section{Finite pointed metric spaces}\label{sec-finite}
We study here the Lip-BPB property for finite pointed metric spaces. The following is our main positive result here.
	
	\begin{teo}\label{finito} Let $M$ be a finite pointed metric space and let $Y$ be a Banach space. If $(\mathcal{F}(M),Y)$ has the BPB property, then $(M,Y)$ has the Lip-BPB property.
	\end{teo}

We will see in Example \ref{example3.10} that the converse of the above result does not hold.

To prove Theorem \ref{finito}, we need to particularize to Lipschitz-free spaces a property which was introduced for general Banach spaces by Schachermayer \cite{Schachermayer} in relation with the density of norm attaining operators.
\begin{defi}\label{def:alpha}
	A Banach space $X$ is said to have \emph{property $\alpha$} if there exist a balanced subset $ \Gamma=\{x_\lambda\}_{\lambda \in \Lambda}$ of $X$, a subset $\Gamma^*=\{x^*_\lambda\}_{\lambda \in \Lambda} \subseteq X^*$, and a constant $0\leq \rho<1$ such that
	\begin{enumerate}[(i)]
			\item $\lVert x_\lambda \rVert = \lVert x^*_\lambda\rVert = \lvert x^*_\lambda(x_\lambda)\rvert =1$ for all $\lambda\in \Lambda$.
			\item $|x^*_\lambda(x_\mu)|\leq \rho$ for all $x_\lambda \neq \pm x_\mu$.
			\item $\overline{\co}\left(\{x_\lambda\colon \lambda \in \Lambda\}\right)= B_X$.
		\end{enumerate}	
\end{defi}

It is shown in \cite[Corollary 3.10]{articulo1} that given a pointed metric space $M$ such that $\mathcal{F}(M)$ has property $\alpha$, then $\SA(M,Y)$ is dense in $\Lip(M,Y)$ for every Banach space $Y$.

	\begin{proof}[Proof of Theorem \ref{finito}]
		Note that Example 3.14 in \cite{articulo1} shows that $\mathcal{F}(M)$ has property $\alpha$. Note that from (ii) we obtain that $\|x_\lambda-x_\mu\|\geq |x_\lambda^*(x_\lambda)-x_\lambda^*(x_\mu)|\geq 1-\rho$ when $x_\lambda\neq \pm x_\mu$. Therefore, as $M$ is finite and so $B_{\mathcal{F}(M)}$ is compact, $\Gamma$ must be a finite set:
		\[ \Gamma = \{ x_k \colon k=1, \ldots, n \}. \]
		Moreover, as $B_{\F(M)} = \cco(\Gamma) = \co(\Gamma)$, every molecule $m_{p,q} \in \Mol(M)$ can be written as a convex combination of these points. Let us take \[\delta=\min\left\{\min\left\{\lambda_k \colon m_{p,q} = \sum_{k=1}^n \lambda_k x_k, \lambda_k > 0\right\} \colon m_{p,q} \in \Mol(M) \right\}>0.\]

		Now, fix $0<\varepsilon<\min\left \{\frac{1}{2}, (1-\rho)\delta \right \}$ and take $\eta(\varepsilon)$ the constant associated to the BPB property of the pair $(\mathcal{F}(M),Y)$. Consider $F \in \Lip(M,Y)$ with $\| F\|_L =1$ and $m \in \Mol(M)$ such that $\|\widehat{F}(m)\|> 1- \eta(\varepsilon)$.
		By hypothesis, there exist $G \in \Lip(M,Y)$ and $\xi \in B_{\F(M)}$ satisfying
		\[ \|\widehat{G}(\xi)\|=\|G \|_L=1, \quad \|F-G\|_L<\varepsilon, \quad \|m-\xi\|< \varepsilon. \]
		Note that we can write
		\[ m= \sum_{k=1}^n \lambda_k x_k, \quad \xi=\sum_{k=1}^n \theta_k x_k, \quad \sum_{k=1}^n \lambda_k = \sum_{k=1}^n \theta_k =1,\quad \lambda_k, \theta_k\geq 0\]
		for every $k=1,\ldots,n$. We claim that  $\lambda_k = 0$ whenever $\theta_k = 0$. Indeed, if we suppose that there exists $k \in \{1,\ldots,n\}$ verifying that $\lambda_k \neq 0$ but $\theta_k =0$, then it makes sense to take the constant $\delta_{\xi,m}$ given by
		\[ \delta_{\xi,m}=\min\left\{ \lambda_k \colon \lambda_k \neq 0,\, \theta_k = 0, \, k=1,\ldots,n \right\}.\]
		Let us consider $j \in \{1,\ldots,n\}$ such that $\lambda_j=\delta_{\xi,m}$, so $\theta_j=0$ and we obtain that
		\begin{align*}
		\|m-\xi\|&\geq x^*_j(m)-x^*_j(\xi) = \sum_{k=1}^n \lambda_k x^*_j(x_k) - \sum_{k=1}^n \theta_k x^*_j(x_k)\\
		&= \lambda_j + \sum_{k\neq j} \lambda_k x^*_j(x_k) - \sum_{k\neq j}\theta_k x^*_j(x_k)\\
		&=\lambda_j -\sum_{k\neq j} (\theta_k - \lambda_k)x^*_j(x_k)\geq \lambda_j - \rho\sum_{k\neq j} (\theta_k - \lambda_k)\\
		&= \lambda_j - \rho(1-(1-\lambda_j))=(1-\rho)\lambda_j = (1-\rho)\delta_{x,m}\geq(1-\rho)\delta>\varepsilon,
		\end{align*}
		a contradiction. Now, taking $y^*\in S_{Y^*}$ such that $y^*(\widehat{G}(\xi))=1$, we have that
		\[ 1=y^*(\widehat{G}(\xi))=\sum_{k=1}^n \theta_k y^*(\widehat{G}(x_k))\leq \sum_{k=1}^n \theta_k=1. \]
		Then, $y^*(\widehat{G}(x_k))=1$ for every $k=1,\ldots,n$ such that $\theta_k\neq 0$. By our assumption, this also happens for every $k=1, \ldots, n$ such that $\lambda_k \neq 0$. Consequently, we have that
		\[ \|\widehat{G}(m)\|\geq y^*(\widehat{G}(m))=\sum_{\lambda_k\neq 0} \lambda_k y^*(\widehat{G}(x_k))=\sum_{\lambda_k\neq 0} \lambda_k =1.\]
		That is, $\widehat{G}$ attains its norm at the molecule $m\in\Mol(M)$.
		\qedhere
	\end{proof}

	It is shown in \cite[Proposition 2.4]{BPBp} that if $X$ and $Y$ are finite-dimensional Banach spaces, then $(X,Y)$ has the BPB property. Consequently, we obtain the following corollary.
	
	\begin{cor}
		Let $M$ be a finite pointed metric space and let $Y$ be a finite-dimensional Banach space. Then, $(M,Y)$ has the Lip-BPB property.
	\end{cor}

	In particular, we obtain the following.
	
	\begin{cor}\label{corfinite}
		Let $M$ be a finite pointed metric space. Then, $(M,\mathbb{R})$ has Lip-BPB property.
	\end{cor}
	
	In \cite[Theorem 2.2]{BPBp} it is also shown that if a Banach space $Y$ has property $\beta$, then the pair $(X,Y)$ has the BPB property for every Banach space $X$. Note that, by using Theorem \ref{finito}, we obtain that given a finite pointed metric space $M$ and a Banach space $Y$ having property $\beta$, the pair $(M,Y)$ will have the Lip-BPB property. In this way we could give more corollaries. However, in Section \ref{sec-fromscalartovectors} we will give a stronger result which generalizes all of them.

	The next example shows that we cannot remove the hypothesis of having  $(\F(M),Y)$ the BPB property in Theorem \ref{finito}. It is an adaption of \cite[Lemma 3.2]{acklm}.
	
	\begin{ejem}\label{finito no}
		Let $M=\{0,1,2\}\subseteq \mathbb{R}$ with the usual metric and let $Y$ be a strictly convex Banach space which is not uniformly convex. Then, $(M,Y)$ fails the Lip-BPB property.
	\end{ejem}
	
	\begin{proof}
		Observe that $\mathcal{F}(M)$ is two-dimensional and that $m_{0,2}=\frac12 m_{0,1} + \frac12 m_{1,2}$, so \[B_{\mathcal{F}(M)}=\overline{\co}\{\pm m_{0,1},\pm m_{1,2}\}\] is a square. On the other hand, as $Y$ is not uniformly convex, there exist sequences $\{x_n\}, \{y_n\} \subseteq S_Y$ and $\varepsilon_0>0$ such that
		\[ \lim_{n \to \infty} \|x_n+y_n\| = 2 \quad \text{ and } \quad \|x_n-y_n\|> \varepsilon_0 \quad \forall \, n \in \mathbb{N}. \]
		Fix $0<\varepsilon<\frac{\varepsilon_0}{2}$ and assume that $(M,Y)$ has the Lip-BPB property witnessed by the function $\varepsilon\longmapsto \eta(\varepsilon)>0$. Take $m \in \mathbb{N}$ such that
		\[\|x_m+y_m\|>2-2\eta(\varepsilon)\]
		and define the linear operator $\widehat{F}\in \operatorname{L}(\F(M), Y)$ by
		\[\widehat{F}(m_{0,1})=x_m, \qquad \widehat{F}(m_{1,2})=y_m. \]
		It is clear, by the shape of the unit ball of $\mathcal{F}(M)$, that $\|\widehat{F}\|=1$. Furthermore, note that
		\[ \|\widehat{F}(m_{0,2})\|=\left\|\widehat{F}\left(\frac{m_{0,1}+m_{1,2}}{2}\right)\right\|=\frac{1}{2}\|x_m+y_m\|>1-\eta(\varepsilon). \]
		Therefore, there exist a linear operator $\widehat{G} \colon \mathcal{F}(M) \longrightarrow Y$ and a molecule $u \in \Mol(M)$ such that
		\[ \|\widehat{G}(u)\|=\|\widehat{G}\|=1, \quad \|\widehat{F}-\widehat{G}\|< \varepsilon, \quad \|m_{0,2}-u\|<\varepsilon. \]
		A straightforward application of Lemma 1.3 in \cite{articulo1} shows that
		\[ \|m_{0,2}-m_{0,1}\|, \|m_{0,2}-m_{1,2}\| \geq 1, \]
		hence $u=m_{0,2}$. Now, note that
		\[ 1=\|\widehat{G}(m_{0,2})\|=\left\|\frac{1}{2}\widehat{G}(m_{0,1})+\frac{1}{2} \widehat{G}(m_{1,2})\right\|.\]
		Since $Y$ is strictly convex, it follows that $\widehat{G}(m_{0,1})=\widehat{G}(m_{1,2})$, which implies that
		\begin{align*}\|x_m-y_m\|&=\|\widehat{F}(m_{0,1})-\widehat{F}(m_{1,2})\|\\ &\leq \|\widehat{F}(m_{0,1})-\widehat{G}(m_{0,1})\| + \|\widehat{F}(m_{1,2}) - \widehat{G}(m_{1,2})\|
		\leq \varepsilon+\varepsilon<\varepsilon_0,
		\end{align*}
		a contradiction.
	\end{proof}

	Finally, the following example shows that the finiteness of the metric space is also necessary in Theorem~\ref{finito}.
	
	\begin{ejem}\label{N}
		$(\mathbb{N},\mathbb{R})$ does not have Lip-BPB property.
	\end{ejem}
	\begin{proof}
		Fix $0<\varepsilon<\frac{1}{2}$ and suppose that $(\N,\R)$ has the Lip-BPB property witnessed by a function $\varepsilon \longmapsto \eta(\varepsilon)>0$ which we can suppose satisfies $\eta(\varepsilon)<\frac{1}{2}$.

		Take $n \in \mathbb{N}$ such that $n>\frac{1}{2\eta(\varepsilon)}$ and define $f \colon \mathbb{N} \longrightarrow \mathbb{R}$ by
		\[ f(p)=  \left\{
		\begin{array}{ll}
		p-1 & \mbox{ if $p\leq 2n$ } \\
		p-2 & \mbox{ if $p > 2n$} \\
		\end{array} \right. \]
		It is clear that $f \in \Lip(\mathbb{N},\mathbb{R})$ with $\|f\|_L=1$. Besides,
		\[ \widehat{f}(m_{3n,n})=\frac{f(3n)-f(n)}{3n-n}= \frac{2n-1}{2n}= 1- \frac{1}{2n}>1-\eta(\varepsilon).\]
		Now, given $p<q \in \mathbb{N}$, if $\hat{g} \in \operatorname{L}(\mathcal{F}(\mathbb{N}),\mathbb{R})$ with $\|g\|_L=1$ attains its norm at a molecule $m_{q,p}$ such that $\|m_{q,p} -m_{3n,n}\|<\varepsilon$, Lemma 1.3 in \cite{articulo1} implies that $[2n, 2n+1]\subseteq [p,q]$. Indeed, if we assume that $p>2n$ or $q<2n+1$, then by applying that lemma we obtain
		\[ \|m_{q,p}-m_{3n,n}\| \geq \frac{\max\{|q-3n|,|p-n|\}}{\min\{|q-p|,2n\}}\geq \frac{n}{2n}=\frac{1}{2},\]
		which is a contradiction since $\varepsilon<\frac{1}{2}$.
		According to \cite[Lemma 2.2]{kms}, $g$ attains its norm at the molecule $m_{2n+1, 2n}$. In view of this, it is enough to note that
		\[ \|g-f\|_L\geq \widehat{g}(m_{2n+1, 2n})-\widehat{f}(m_{2n_0+1, 2n})= 1-0=1, \]
		which is a contradiction.
	\end{proof}
	
	\section{Universal Lip-BPB property metric spaces}\label{sec-universal}
	
	A pointed metric space $M$ is said to be \textit{concave} if every molecule of $M$ is a preserved extreme point, that is, an extreme point of the unit ball of $\mathcal{F}(M)^{**}$. This property has been recently characterized in \cite{ag} and, for a boundedly compact pointed metric space $M$ it is shown in \cite[Proposition 3.34]{wea2} that it is equivalent to the fact that $$d(x,y) < d(x,z)+d(z,y)$$ for all distinct points $x,y,z\in M$ (recall that a metric space $M$ is said to be \emph{boundedly compact} if every bounded closed subset of $M$ is compact). In fact, it is proved in \cite[Theorem 1.1]{ap} that a molecule $m_{x,y}$ is an extreme point of the unit ball of $\mathcal{F}(M)$ if, and only if, the above inequality holds for every point $z\in M\setminus\{x,y\}$. A strengthening of the concept of concavity is provided when we require all the molecules to be strongly exposed points of the unit ball of $\F(M)$. By the characterization given in \cite[Theorem 5.4]{gpr}, the property can be written in terms of the metric space and we may also introduce a uniform version of it. We need some notation. Given $x, y, z \in M$, the \emph{Gromov product} \cite[p.~410]{bh} of $x$ and $y$ at $z$ is defined as
	\[(x,y)_z:=\frac{1}{2}\bigl(d(x,z)+d(z,y)-d(x,y)\bigr)\geq 0.\]
	It corresponds to the distance of $z$ to the unique closest point $b$ on the unique geodesic between $x$ and $y$ in any $\mathbb{R}$-tree into which $\{x,y,z\}$ can be isometrically embedded (such a tree always exists). If $X$ is a Banach space, given $f \in S_{X^*}$ and $0<\delta<1$, the \emph{slice} of $B_X$ associated to $f$ and $\delta$ is the set 	
\[ S(B_X,f,\delta)=\{x \in B_{X} \colon f(x)> 1-\delta\}.\]

\begin{defi}
Let $M$ be a pointed metric space.
\begin{enumerate}
  \item[(a)] We say that $M$ is \emph{Gromov concave} if for every $x,y\in M$, $x\neq y$, there is $\varepsilon_{x,y}>0$ such that
  $$(x,y)_z > \varepsilon_{x,y} \min\{d(x,z), d(y,z)\}$$ for every $z\in M\setminus\{x,y\}$.
  \item[(b)] We say that $M$ is \emph{uniformly Gromov concave} if there is $\varepsilon_0>0$ such that $$(x,y)_z > \varepsilon_0 \min\{d(x,z), d(y,z)\}$$ for every distinct $x,y,z\in M$.
\end{enumerate}
\end{defi}

By \cite[Theorem 5.4]{gpr}, $M$ is Gromov concave if and only if every molecule is a strongly exposed point of the unit ball of $\F(M)$. By \cite[Proposition 3.6]{articulo1}, $M$ is uniformly Gromov concave if this happens uniformly:

\begin{remark}\label{remark-Gromovuniformlyconcave}
A pointed metric space $M$ is uniformly Gromov concave if and only if $\Mol(M)$ is a set of uniformly strongly exposed points of $B_{\F(M)}$, that is, there exists a family $\{f_m\}_{m \in \Mol(M)} \subseteq \Lip(M,\R)\equiv \F(M)^*$ with $\|f_m\|_L=\widehat{f}_m(m)=1$ such that for every $\varepsilon>0$ there is $\delta>0$ satisfying that
\[
\diam \bigl( S(B_{\mathcal{F}(M)}, \widehat{f}_m, \delta) \bigr)<\varepsilon \quad \forall \, m \in \Mol(M).
\]
\end{remark}

In the notation of \cite[Definition 3.5]{articulo1}, $M$ is uniformly Gromov concave if and only if $\Mol(M)$ is uniformly Gromov rotund (the more concave is the metric space $M$, the more ``rotund'' is $\Mol(M)$).

It is shown in Proposition 3.3 of \cite{articulo1} that if $B_{\mathcal{F}(M)}$ is the closed convex hull of a set of uniformly strongly exposed points, then $\SA(M,Y)$ is dense in $\Lip(M,Y)$ for every Banach space $Y$. If such a set is the whole $\Mol(M)$ (that is, if $M$ is uniformly Gromov concave), we actually get more.
	
	\begin{teo}\label{theorem-Gromovunifconcave}
 Let $M$ be a uniformly Gromov concave pointed metric space. Then, $(M,Y)$ has the Lip-BPB property for every Banach space $Y$.
	\end{teo}
	
	\begin{proof}
Fix $0<\varepsilon<1$. Since $\Mol(M)$ is a set of uniformly strongly exposed points (Remark \ref{remark-Gromovuniformlyconcave}), there exists $0<\delta<1$ such that
\begin{equation}\label{eq:gromovunifconcave_1}
\diam \bigl( S(B_{\mathcal{F}(M)}, \widehat{f}_m, \delta) \bigr) <\varepsilon \quad \forall \, m \in \Mol(M),
\end{equation}
		where $\{\widehat{f}_m\}_{m \in \Mol(M)}$ are the functionals which uniformly strongly expose the molecules of $M$. We take $\eta>0$ satisfying
		\[ \left(1+\frac{\varepsilon}{4}\right)(1-\eta) > 1+\frac{\varepsilon (1-\delta)}{4}.\]
		Now, consider $F \in \Lip(M,Y)$ with $\|F\|_L=1$ and a molecule $m \in \Mol(M)$ such that $\|\widehat{F}(m)\|> 1-\eta$. Then, we define $\widehat{G}_0 \in \operatorname{L}(\F(M),Y)$ given by
		\[ \widehat{G}_0(x)=\widehat{F}(x)+\frac{\varepsilon}{4} \widehat{f}_m(x)\widehat{F}(m) \quad \forall \, x \in \F(M).\]
		It is clear that $\|\widehat{F}-\widehat{G}_0\|\leq \frac{\varepsilon}{4}$. In addition, note that
		\[ \|\widehat{G}_0(m)\|=\left(1+\frac{\varepsilon}{4}\right)\|\widehat{F}(m)\| \geq \left(1+\frac{\varepsilon}{4}\right)(1-\eta).\]
		On the other hand, if $ x  \notin \pm S(B_{\F(M)},f_m,\delta)$ and $\|x\|\leq1$, then we will have that
		\[ \|\widehat{G}_0(x)\|=\left\|\widehat{F}(x)+\frac{\varepsilon}{4}\widehat{f}_m(x)\widehat{F}(m)\right\|\leq 1+\frac{\varepsilon}{4}|\widehat{f}_m(x)| \leq 1+\frac{\varepsilon (1-\delta)}{4}. \]
		Therefore, $\|\widehat{G}_0(x)\|\geq \|\widehat{G}_0(m)\|$ implies that $x \in \pm S(B_{\F(M)},\widehat{f}_m,\delta)$. By defining $\widehat{G}=\frac{\widehat{G}_0}{\|\widehat{G}_0\|}$, we have that
		\[\|\widehat{F}-\widehat{G}\|\leq \|\widehat{F}-\widehat{G}_0\|+\|\widehat{G}_0-\widehat{G}\|=\|\widehat{F}-\widehat{G}_0\|+\bigl|\|\widehat{G}_0\|-1\bigr| \leq \frac{\varepsilon}{4}+\frac{\varepsilon}{4}=\frac{\varepsilon}{2}.\]
	 	Note that if $\widehat{G}$ attains its norm at the molecule $m$, then we have finished. Otherwise, we may take $\varepsilon'$ with $0<\varepsilon'<\min\{\frac{\varepsilon}{2}, \|\widehat{G}\|-\|\widehat{G}(m)\|\}$.
		Now, thanks to Proposition 3.3 in \cite{articulo1}, $\SA(M,Y)$ is dense in $\Lip(M,Y)$. Hence, there exist $\widehat{H} \in \operatorname{L}(\F(M),Y)$ and $u \in \Mol(M)$ satisfying
		\[ \|\widehat{H}\|=\|\widehat{H}(u)\|=1 \quad \mbox{ and } \quad \|\widehat{G}-\widehat{H}\|<\varepsilon'.\]
		Next, we note that
		\[ \|\widehat{G}(u)\|\geq \|\widehat{H}(u)\|-\|\widehat{H}-\widehat{G}\| \geq \|\widehat{H}\|-\varepsilon' \geq \|\widehat{H}\|-(\|\widehat{G}\|-\|\widehat{G}(m)\|)=\|\widehat{G}(m)\|,\]
		which implies that $\|\widehat{G}_0(u)\|\geq\|\widehat{G}_0(m)\|$, hence $u \in \pm S(B_{\mathcal{F}(M)},f_m,\delta)$. It follows from \eqref{eq:gromovunifconcave_1} that $\|m-u\|<\varepsilon$ or $\|m+u\|<\varepsilon$. Finally, note that
		\[ \|\widehat{F}-\widehat{H}\|\leq \|\widehat{F}-\widehat{G}\|+\|\widehat{G}-\widehat{H}\| <\varepsilon. \qedhere\]
	\end{proof}
	
	This result produces interesting corollaries. First, if $M$ is concave and $\mathcal{F}(M)$ has property $\alpha$ (see Definition \ref{def:alpha}), then $M$ is uniformly Gromov concave by \cite[Theorem 3.16]{articulo1}. Therefore, we obtain the next corollary.
	
	\begin{cor}\label{alpha}
		Let $M$ be a concave pointed metric space such that $\mathcal{F}(M)$ has property $\alpha$. Then, $(M,Y)$ has the Lip-BPB property for every Banach space $Y$.
	\end{cor}

	Note that the concavity hypothesis in the previous result is necessary as Example \ref{finito no} and \cite[Example 3.15.b]{articulo1} show.
	
	Since for every finite pointed metric space, $\mathcal{F}(M)$ has property $\alpha$ (see \cite[Example 3.15.a]{articulo1}), we obtain the following interesting particular case.
	
	\begin{cor}\label{alphafinito}
		Let $M$ be a concave finite pointed metric space. Then $(M,Y)$ has the Lip-BPB property for every Banach space $Y$.
	\end{cor}

Another class of metric spaces for which Theorem \ref{theorem-Gromovunifconcave} is applicable is the ultrametric spaces. A metric space is said to be \emph{ultrametric} if the inequality $$d(x,y)\leq \max\bigl\{d(x,z),d(z,y)\bigr\}$$ holds for all $x,y,z\in M$. This class of metric spaces has been deeply studied due to its relations with the problem of finding good embedding of metric spaces, see \cite{Naor} and references therein, for instance. Properties on the Lipschitz-free space over an ultrametric space can be found in \cite{CuthDoucha} and references therein, for instance. It readily follows that every ultrametric space is uniformly Gromov concave, so we get the following consequence of Theorem \ref{theorem-Gromovunifconcave}.

\begin{cor}
If $M$ is a pointed ultrametric space, then $(M,Y)$ has the Lip-BPB property for every Banach space $Y$.
\end{cor}

Finally, we may also obtain a large class of metric spaces, which includes connected metric spaces, for which the Lip-BPB property is satisfied for every Banach space $Y$: the class of H\"{o}lder metric spaces. We refer the reader to the paper \cite{Kalton04} and the book \cite{wea2} as good references on H\"{o}lder metric spaces.

	\begin{cor}
		Let $M$ be a H\"{o}lder pointed metric space. Then, $(M,Y)$ has the Lip-BPB property for every Banach space $Y$.
	\end{cor}

The result follows from Theorem \ref{theorem-Gromovunifconcave} by using the following proposition.

\begin{prop}\label{holder}
Every H\"{o}lder metric space is uniformly Gromov concave.
\end{prop}

\begin{proof}
Let $(M,d)$ be a metric space and fix $0<\theta<1$. Consider $\varepsilon_0=1-2^\theta$ and let us show that $(M,d^\theta)$ is uniformly Gromov concave witnessed by $\frac{\varepsilon_0}{2}$. Indeed, given $t>0$ define $f_t \colon [0,t) \longrightarrow \mathbb{R}$ by
\[
f_t(s)=\frac{t^\theta-s^\theta}{(t-s)^\theta}\quad \forall\, s \in [0,t).
\]
	It is easy to see that $f_t$ is strictly decreasing. Besides, for every $t>0$ we have that
	\[
f_t\left (\frac{t}{2}\right )=\frac{t^\theta-(\frac{t}{2})^\theta}{(\frac{t}{2})^\theta}=2^\theta -1.\]
	Take $x, y, z$ distinct points of $M$. We may assume that $d(x,z)\leq d(y,z)$. Consequently, we have that $d(y,z)\geq \frac{d(x,y)}{2}$. We distinguish two cases:

\noindent (1): $d(x,y)>d(y,z)$. In this case, we estimate
		\begin{align*}
\frac{d(x,z)^\theta + d(y,z)^\theta - d(x,y)^\theta}{d(x,z)^\theta} &= 1-\frac{d(x,y)^\theta - d(y,z)^\theta}{d(x,z)^\theta} \\ &= 1-\frac{d(x,y)^\theta - d(y,z)^\theta}{(d(x,y)-d(y,z))^\theta} \frac{(d(x,y)-d(y,z))^\theta}{d(x,z)^\theta} \\ &\geq 1-f_{d(x,y)}(d(y,z))\frac{d(x,z)^\theta}{d(x,z)^\theta}\geq 1-f_{d(x,y)}\left (\frac{d(x,y)}{2}\right )=2-2^\theta.
		\end{align*}
\noindent (2): $d(x,y)\leq d(y,z)$. Here it is enough to note that
		\[  \frac{d(x,z)^\theta + d(y,z)^\theta - d(x,y)^\theta}{d(x,z)^\theta}\geq \frac{d(x,z)^\theta}{d(x,z)^\theta}=1.\qedhere\]
\end{proof}

Let us comment that a particular case of \cite[Example~3.38]{wea2} is that every H\"{o}lder metric space is concave (uniform concave in the notation of that book). Thus, the above proposition improves such a particular case.

	A natural question is the following: does there exist any relationship between the BPB property of $(\mathcal{F}(M),Y)$ and the Lip-BPB property of $(M,Y)$?
	Example \ref{N} partially answers this question in a negative way. Note that, since the Bishop-Phelps-Bollob\'as theorem is valid for every Banach space, we know that the pair $(\mathcal{F}(\mathbb{N}),\mathbb{R})$ has the BPB property. However, in that Example it is shown that $(\mathbb{N},\mathbb{R})$ fails the Lip-BPB property, so the BPBp of $(\mathcal{F}(M),Y)$ does not imply the Lip-BPB property of $(M,Y)$ in general. The next result will show that the Lip-BPB property of $(M,Y)$ does not imply the BPB property of ($\mathcal{F}(M),Y)$.
	
	\begin{prop}\label{finitofalla}
		Let $M$ be a finite pointed metric space with more than two points Then, there exists a Banach space $Y$ such that $(\F(M),Y)$ fails the BPB property.
	\end{prop}
	
	\begin{proof}
Assume that $(\mathcal{F}(M),Y)$ has the BPB property for every Banach space $Y$. Being finite-dimensional, $\F(M)$ is isomorphic to a strictly convex Banach space. Then, by \cite[Corollary 3.5]{acklm}, the set of extreme points of $B_{\F(M)}$ is dense in $S_{\F(M)}$. However, we have that $B_{\mathcal{F}(M)}=\co\bigl(\Mol(M)\bigr)$ since $\Mol(M)$ is finite hence compact, so every extreme point of $B_{\F(M)}$ has to be contained in $\Mol(M)$. But, being finite, the set $\Mol(M)$ cannot be dense in $S_{\F(M)}$ if $M$ contains more than two points.
	\end{proof}

It can be deduced from \cite[Theorem 3.2]{km} and Lemma \ref{recdomain} that the space $Y$ above can be taken as a $c_0$-sum of some concrete strictly convex renormings of $\mathcal{F}(M)$.

\begin{ejem}\label{example3.10}
If we consider a concave finite pointed metric space $M$, then $(M,Y)$ has the Lip-BPB property for every Banach space $Y$ by Corollary \ref{alphafinito}, while we may consider a Banach space $Y$ such that $(\mathcal{F}(M),Y)$ fails the BPB property thanks to Proposition \ref{finitofalla}.
\end{ejem}
	
Note that Examples \ref{finito no} and \ref{N} show that it seems like the Lip-BPB property does not hold when the metric space has many nontrivial metric segments. For this reason, and in view of Corollary \ref{alpha}, we could believe that if the metric space is concave or even Gromov concave, $(M,Y)$ may have the Lip-BPB property for all Banach spaces $Y$. However, the next example shows that this does not always happen, even for scalar Lipschitz functions.

\begin{ejem}
There exists a Gromov concave pointed metric space such that $\F(M)$ has the RNP and $(M,\mathbb{R})$ fails the Lip-BPB property.
\end{ejem}

	\begin{proof}
		Let us consider $M=\left\{\left(n,\frac{1}{n^2}\right)\colon n\in \mathbb{N}\right\} \subseteq \mathbb{R}^2$ with the Euclidean metric. This metric space is boundedly compact and every metric segment is trivial, so $M$ is concave by \cite[Proposition 3.34]{wea2}. Furthermore, since $M$ is uniformly discrete, Proposition 5.3 in \cite{lppr} gives that $M$ is Gromov concave. In addition, uniformly discreteness also implies that $\mathcal{F}(M)$ has the RNP \cite[Proposition 4.4]{Kalton04}. We will write $\overline{n}$ to refer to the point $\left(n,\frac{1}{n^2}\right)$ for every $n \in \mathbb{N}$. Fix $0<\varepsilon<\frac{1}{3}$ and suppose that $(M,\R)$ has the Lip-BPB property witnessed by the function $\varepsilon \longmapsto \eta(\varepsilon)$, which we may suppose satisfies $0<\eta(\varepsilon)<\frac{1}{3}$.

		For every $n \in \mathbb{N}$, we define $f_n \colon M \longrightarrow \mathbb{R}$ by
		\[ f_n(\overline{p})=  \left\{
		\begin{array}{ll}
		p-1 & \mbox{ if $p\leq 2n$ } \\
		p-2 & \mbox{ if $p > 2n$} \\
		\end{array} \right. \]
		It is clear that $f_n \in \Lip(M,\mathbb{R})$ and $\|f_n\|_L\leq 1$. Furthermore, given $k >2n$ we have that
		\[ \widehat{f}_n(m_{\overline{k+1},\overline{k}})=\frac{f_n(\overline{k+1})-f_n(\overline{k})}{d(\overline{k+1},\overline{k})} = \frac{1}{\sqrt{1+\left(\frac{1}{k^2} - \frac{1}{(k+1)^2}\right)^2}}, \]
		from which we deduce that $\lim\limits_{k \to \infty} \widehat{f}_n(m_{\overline{k+1},\overline{k}}) = 1$ and so $\|f_n\|_L=1$. Now, let us estimate the value of $\widehat{f}_n$ at the molecule $m_{\overline{3n},\overline{n}}$:
		\[ \widehat{f}_n(m_{\overline{3n},\overline{n}})=\frac{f_n(\overline{3n})-f_n(\overline{n})}{d(\overline{3n},\overline{n})}=\frac{2n-1}{2n} \frac{2n}{\sqrt{(2n)^2+\left(\frac{1}{n^2} - \frac{1}{(3n)^2}\right)^2}}. \]
		Therefore, there exists $n_0 \in \mathbb{N}$ such that for every $n \geq n_0$ we have that $\widehat{f}_{n}(m_{\overline{3n},\overline{n}}) > 1-\eta(\varepsilon)$. Now, the Lip-BPB property of $(M,\R)$ gives $g_n \in \Lip(\mathbb{N},\mathbb{R})$ and a molecule $m_{\overline{p_n},\overline{q_n}}$ such that $$\|g_n\|_L=|\widehat{g}_n(m_{p_n,q_n})|=1, \quad \|f_n-g_n\|_L<\varepsilon, \quad \|m_{\overline{p_n},\overline{q_n}} -m_{\overline{3n},\overline{n}}\|<\varepsilon.$$ Note that since $f_n$ is increasing, $p_n$ must be greater than $q_n$. As we did in the proof of Example \ref{N}, by applying \cite[Lemma 1.3]{articulo1} we obtain that $[2n, 2n+1]\subseteq [q_n,p_n]$.
On the one hand, we have that
		\[\widehat{f}_n(m_{\overline{2n+1},\overline{2n}})=0.\]
		On the other hand, by using Lemma 3.7 in \cite{articulo1}, it follows that

		\[\widehat{g}_n(m_{\overline{p_n},\overline{2n+1}})\geq 1-2\frac{(\overline{p_n},\overline{q_n})_{\overline{2n+1}}}{d(\overline{2n+1},\overline{p_n})}\geq 1-\frac{1}{(2n+1)^2d(\overline{2n+1},\overline{p_n})},
\]
		which implies that
		\begin{align*} \widehat{g}_n(m_{\overline{2n+1},\overline{2n}})&=\widehat{g}_n(m_{\overline{p_n},\overline{2n+1}})\frac{d(\overline{2n+1},\overline{p_n})}{d(\overline{2n},\overline{2n+1})} - \widehat{g}_n(m_{\overline{p_n},\overline{2n}})\frac{d(\overline{2n},\overline{p_n})}{d(\overline{2n},\overline{2n+1})}\\
		& \geq \widehat{g}_n(m_{\overline{p_n},\overline{2n+1}})\frac{d(\overline{2n+1},\overline{p_n})}{d(\overline{2n},\overline{2n+1})} - \frac{d(\overline{2n},\overline{p_n})}{d(\overline{2n},\overline{2n+1})}\\
		& \geq \frac{(2n+1)^2d(\overline{2n+1},\overline{p_n}) - 1 - (2n+1)^2d(\overline{2n},\overline{p_n})}{(2n+1)^2d(\overline{2n},\overline{2n+1})}\\
		&\geq \frac{d(\overline{2n+1},\overline{p_n})-d(\overline{2n},\overline{p_n})}{d(\overline{2n},\overline{2n+1})} - \frac{1}{(2n+1)^2}.
		\end{align*}	
		A simple calculation shows that we may take $n_1>n_0 \in \mathbb{N}$ such that $\widehat{g}_n(m_{\overline{2n+1},\overline{2n}})\geq\frac{1}{2}$ for every $n \geq n_1$. Finally, for $n \geq n_1$ observe that
		\[ \|g_n-f_n\|_L\geq \widehat{g}_n(m_{\overline{2n+1},\overline{2n}})- \widehat{f}_n(m_{\overline{2n+1},\overline{2n}}) \geq \frac{1}{2}-0=\frac{1}{2},\]
		a contradiction.
	\end{proof}	

	\section{From scalar functions to vector-valued maps and viceversa}\label{sec-fromscalartovectors}
Our aim here is to show when we may pass the Lip-BPB property for vector-valued maps to the Lip-BPB property for scalar functionals and, conversely, from the Lip-BPB property for scalar functionals to the Lip-BPB property for some vector-valued maps. In the first case, the result is optimal.

	\begin{prop}\label{miguel bollobas}
		Let $M$ be a pointed metric space. Suppose that there exists a Banach space $Y\neq 0$ such that $(M,Y)$ has the Lip-BPB property. Then, $(M,\mathbb{R})$ has the Lip-BPB property.
	\end{prop}
	
	\begin{proof}
		Let $Y$ be a Banach space such that $(M,Y)$ has the Lip-BPB property. Fix $\varepsilon>0$ and consider $\eta(\varepsilon)$ the constant associated to the Lip-BPB property of $(M,Y)$. Let us consider $f \in \Lip(M,\mathbb{R})$ with $\|f\|_L=1$ and $m \in \Mol(M)$ such that $\widehat{f}(m)>1-\eta(\frac{\varepsilon}{2})$. Pick $y_0\in S_Y$ and define $F\in \Lip(M,Y)$ by
		\[ F(p)=f(p)y_0 \quad \forall \, p \in M. \]
Then, we have that $\|F\|_L=1$ and $\|\widehat{F}(m)\|>1-\eta(\frac{\varepsilon}{2})$. So, by hypothesis, there exist $G \in \Lip(M,Y)$ and $u \in \Mol(M)$ satisfying that
		\[ \|\widehat{G}(u)\|=\|G\|_L=1, \quad \|F-G\|_L< \frac{\varepsilon}{2}, \quad \|m-u\|<\frac{\varepsilon}{2}. \]
		Now, take $y^* \in S_{Y^*}$ such that $y^*\bigl (\widehat{G}(u)\bigr )=1$ and note that
		\[\|y^*(y_0)f-y^*\circ G\|_L=\|y^*\circ F- y^*\circ G\|_L\leq \|y^*\| \|F-G\|_L<\frac{\varepsilon}{2}. \]
		This implies that
		\[ y^*(y_0)\geq y^*(y_0)\widehat{f}(u)\geq y^*(\widehat{G}(u))- |y^*(y_0)\widehat{f}(u) - y^*(\widehat{G}(u))| \geq 1-\frac{\varepsilon}{2}. \]
		Therefore, writing $g=y^*\circ G \in \Lip(M,\mathbb{R})$, we have that
		\begin{equation*}
|\widehat{g}(u)|=\|g\|_L=1,\qquad \|g-f\|_L \leq \|g-y^*(y_0)f\|_L+\|y^*(y_0)f-f\|_L<\frac{\varepsilon}{2} + \frac{\varepsilon}{2} =\varepsilon.
		\end{equation*}
As we already know that $\|m-u\|<\varepsilon$, we have that $(M,\mathbb{R})$ has the Lip-BPB property.
	\end{proof}

	We may state an analogous result for the density of strongly norm attaining Lipschitz maps.
	
	\begin{prop}\label{miguel} Let $M$ be a pointed metric space. Suppose that there exists a Banach space $Y\neq 0$ such that $\SA(M,Y)$ is norm-dense in $\Lip(M,Y)$. Then, \[\overline{\SA(M,\mathbb{R})}=\Lip(M,\mathbb{R}).\]
	\end{prop}
	
	\begin{proof}
		Fix $\varepsilon>0$ and consider $f \in \Lip(M,\mathbb{R})$, which we may assume to have norm one. If we define $F$ as in Proposition \ref{miguel bollobas} then, by hypothesis, there exist $G \in \SA(M,Y)$ and $m \in \Mol(M)$ satisfying
		\[ \|\widehat{G}(m)\|=\|G\|_L=1, \quad \|F-G\|_L<\frac{\varepsilon}{2}. \]
		Taking $y^* \in S_{Y^*}$ such that $y^*(\widehat{G}(m))=1$ and repeating the argument used in Proposition \ref{miguel bollobas}, we obtain that
		\[\|y^*\circ G-y^*(y_0)f\|_L<\frac{\varepsilon}{2},\quad y^*(y_0)\geq 1-\frac{\varepsilon}{2}. \]
		Consequently, we have that $g=y^*\circ G \in \Lip(M,\mathbb{R})$ satisfies that
		\[\|g - f\|_L \leq \|g - y^*(y_0)f\|_L+\|y^*(y_0)f-f\|_L < \frac{\varepsilon}{2} + |1-y^*(y_0)|<\varepsilon,\]
		and that $\widehat{g}(m)=\|g\|_L=1$, so $g\in \SA(M,\mathbb{R})$.
	\end{proof}

	Our next aim in this section is to study the converse problem of passing from the Lip-BPB property for scalar functionals to some vector-valued maps. Actually, we present a sufficient condition on $Y$ assuring that a pair $(M,Y)$ has the Lip-BPB property when $(M,\mathbb{R})$ does. The following definition is introduced in \cite{lindens} by Lindenstrauss.
	
	\begin{defi}\label{beta}
		Let $Y$ be a Banach space. We will say that $Y$ has \emph{property $\beta$} if there is a set $\{(y_\lambda^*,y_\lambda)\colon \lambda \in \Lambda\} \subset Y^*\times Y$, and a constant $0\leq \rho<1$ satisfying
		\begin{enumerate}
			\item $\|y_\lambda^*\|=\|y_\lambda\|=y_\lambda^*(y_\lambda)=1$ for every $\lambda \in \Lambda$.
			\item $|y^*_\lambda (y_\mu)|\leq \rho$ for every $\lambda \neq \mu \in \Lambda$.
			\item $\|y\|=\sup\{|y_\lambda^*(y)|\colon \lambda \in \Lambda \}$ for every $y \in Y$.
		\end{enumerate}
		\end{defi}
	
Examples of Banach spaces with property $\beta$ are the finite-dimensional spaces whose unit ball is a polyhedron and those spaces $Y$ such that $c_0\subset Y \subset \ell_\infty$ (canonical copies). Besides, Partington proved in \cite{Partington} that every Banach space can be renormed to satisfy property $\beta$. It is convenient to comment that this property $\beta$ is somehow dual to property $\alpha$ (see \cite[Proposition 1.4]{Schachermayer}). The proof of the next result is based on \cite[Theorem 2.2]{BPBp}.

	\begin{prop}\label{propbeta}
		Let $M$ be a pointed metric space such that $(M,\mathbb{R})$ has the Lip-BPB property and let $Y$ be a Banach space satisfying property $\beta$. Then, $(M,Y)$ has the Lip-BPB property.
	\end{prop}

	\begin{proof}
Suppose that $(M,\R)$ has the Lip-BPB property witnessed by a function $\varepsilon\longmapsto \eta'(\varepsilon)$.
		Fix $\varepsilon>0$ and us consider $0<\gamma<\frac{\varepsilon}{2}$ satisfying
		\[ 1+\rho\left ( \frac{\varepsilon}{2}+\gamma\right )<\left (1+\frac{\varepsilon}{2}\right )(1-\gamma).\]
		Consider $\widehat{F}\in \operatorname{L}(\F(M),Y)$ with $\|F\|_L=1$, $m \in \Mol(M)$ such that $\|\widehat{F}(m)\|>1-\eta'(\gamma)$. Take $\lambda \in \Lambda$ such that $|y^*_\lambda(\widehat{F}(m))|>1-\eta(\gamma)$. By hypothesis, there exist $\widehat{g} \in \F(M)^*$ and $u \in \Mol(M)$ such that
		\[
		\|\widehat{g}(u)\|=\|\widehat{g}\|=\|\widehat{F}^*(y^*_\lambda)\|>1-\gamma, \quad \|\widehat{g} -\widehat{F}^*(y^*_\lambda)\|<\gamma, \quad \|m-u\|<\gamma.
		\]
		Define the operator $\widehat{G} \in \operatorname{L}(\mathcal{F}(M),Y)$ by
		\[ \widehat{G}(x)= \widehat{F}(x)+ \left [\left (1+\frac{\varepsilon}{2}\right )\widehat{g}(x)-\widehat{F}^*(y_\lambda^*)(x)\right ]y_\lambda\quad \forall \, x \in \mathcal{F}(M). \]
		Then, we have that
		\[ \|\widehat{G} -\widehat{F}\|\leq \frac{\varepsilon}{2}\|\widehat{g}\|+\|\widehat{g}-\widehat{F}^*(y_\lambda^*)\|\leq \frac{\varepsilon}{2}+\gamma<\varepsilon.\]
		Therefore, we will finish if we prove that $\widehat{G}$ attains its norm at $u \in \Mol(M)$. Since for every $y^* \in Y^*$ one has
		\[ 	\widehat{G}^*(y^*)=\widehat{F}^*(y^*)+y^*(y_\lambda)\left (\frac{\varepsilon}{2}\widehat{g}+\widehat{g}-\widehat{F}^*(y^*_\lambda)\right ),\]
		given $\mu\in \Lambda$, $\mu\neq \lambda$, we have that
		\[ \|\widehat{G}^*y_\mu^*\|\leq 1+\rho\left(\frac{\varepsilon}{2}+\gamma \right) \leq 1+\rho\left (\frac{\varepsilon}{2}+\gamma\right ).\]
		On the other hand, $\widehat{G}^*(y_\lambda^*)=\left (1+\frac{\varepsilon}{2}\right )\widehat{g}$, so
		\[ \|\widehat{G}^*(y_\lambda^*)\|=\left (1+\frac{\varepsilon}{2}\right )\|\widehat{g}\|>\left (1+\frac{\varepsilon}{2}\right )(1-\gamma)>1+\rho\left (\frac{\varepsilon}{2}+\gamma\right ).	\]
		Consequently, $\|\widehat{G}^*\|=\|\widehat{G}^*(y_\lambda^*)\|$, but $\widehat{G}^*(y_\lambda^*)$ is a multiple of $\widehat{g}$, so it attains its norm as a functional on $\mathcal{F}(M)$ at $u$, hence $\widehat{G}$ attains its norm at the molecule $u \in \Mol(M)$, as desired.
	\end{proof}

\begin{remark}\label{remark:betaonlyrho}
It can be shown from the above proof that when $(M,\mathbb{R})$ has the Lip-BPB property and $Y$ is a Banach space satisfying property $\beta$ with constant $\rho$, then the pair $(M,Y)$ has the Lip-BPB property witnessed by a function $\varepsilon \longmapsto \eta(\varepsilon)$ which only depends on $M$ and on $\rho$, not on the particular space $Y$.
\end{remark}

	We could obtain an analogous result for the density of $\SA(M,Y)$. In fact, it is possible to get a more general result in this case. In \cite{aap}  a property called quasi-$\beta$ for a Banach space $Y$ is introduced as a property weaker than property $\beta$ which still implies that $\overline{\NA(X,Y)}=\operatorname{L}(X,Y)$ for every Banach space $X$.
	
	\begin{defi}
		We will say that a Banach space $Y$ has property quasi-$\beta$ if there exist a subset $A \subset S_{Y^*}$, a mapping $\sigma \colon A \longrightarrow S_Y$, and a real-valued function $\rho$ on $A$ satisfying the following conditions:
		\begin{enumerate}
			\item $y^*(\sigma(y))=1$ for every $y^* \in A$.
			\item $|z^*(\sigma(y^*))|\leq \rho(y^*)<1$ for every $y^*,z^* \in A$, $y^*\neq z^*$.
			\item For every extreme point $e^*$ in the unit ball of $Y^*$, there is a subset $A_{e^*}$ of $A$ and a scalar $t$ with $|t|=1$	such that $te^*$ lies in the $w^*$-closure of $A_{e^*}$ and $\sup\{\rho(y^*)\colon y^*\in A_{e^*}\}<1$.
		\end{enumerate}
	\end{defi}

	Every Banach space having property $\beta$ will also have property quasi-$\beta$. Moreover, property quasi-$\beta$ is stable under $c_0$-sums (see \cite[Proposition 4]{aap}), so $c_0$-sums of Banach spaces having property $\beta$ have property quasi-$\beta$, but may have not property $\beta$. In addition, there are finite-dimensional Banach spaces having property quasi-$\beta$ but not $\beta$ (see \cite[Example 5]{aap}).

	The next result, based on the proof of Theorem 2 in \cite{aap}, shows that this property also implies the density of strongly norm-attaining Lipschitz maps from the density in the scalar case.

	\begin{prop}\label{quasi-beta}
		Let $M$ be a pointed metric space such that $\SA(M,\mathbb{R})$ is norm dense in $\Lip(M,\mathbb{R})$ and let $Y$ be a Banach space having property quasi-$\beta$. Then, we have that
		\[ \overline{\SA(M,Y)}=\Lip(M,Y).\]
	\end{prop}

	\begin{proof}
		First, we use a result of V. Zizler in \cite{zizler} which states that the set
		\[
	\{T\in \operatorname{L}(X,Y) \colon T^* \in \NA(Y^*,X^*)\}
	\]
	is dense in $\operatorname{L}(X,Y)$ for every Banach spaces $X$ and $Y$. Therefore, it will be enough to show that for every $\widehat{F}\in \operatorname{L}(\F(M),Y)$ with $\|F\|_L=1$ in this set and $\varepsilon>0$ there exist $\widehat{G} \in \operatorname{L}(\F(M),Y)$ and $u \in \Mol(M)$ such that
		\[
\|\widehat{G}(u)\|=\|G\|_L=1 \quad \mbox{ and } \quad \|F-G\|_L<\varepsilon.
\]
		By a result of T.~Johannesen (see \cite[Theorem 5.8]{lima}), we know that $\widehat{F}^*$ attains its norm at an extreme point $e^*$ of $B_{Y^*}$, and the definition of property quasi-$\beta$ gives us a set $A_{e^*}\subseteq A$ and a scalar $t$ with $|t|=1$ such that $te^*$ lies in the $w^*$-closure of $A_{e^*}$ and
		\[
r=\sup\{\rho(y^*)\colon y^* \in A_{e^*}\}<1.
\]
		Let us fix $0<\gamma<\frac{\varepsilon}{2}$ satisfying
		\[ 1+r\left ( \frac{\varepsilon}{2}+\gamma\right )<\left (1+\frac{\varepsilon}{2}\right )(1-\gamma)\]
		and take $y_1^* \in A_{e^*}$ such that $\|\widehat{F}^*y^*_1\|>1-\gamma$. By hypothesis, there exist $\widehat{g} \in \F(M)^*$ and $u \in \Mol(M)$ such that
		\[
\|\widehat{g}(u)\|=\|\widehat{g}\|=\|\widehat{F}^*(y^*_1)\|>1-\gamma \quad \mbox{ and } \quad \|\widehat{g} -\widehat{F}^*(y^*_1)\|<\gamma.
\]
		Define the operator $\widehat{G} \in \operatorname{L}(\mathcal{F}(M),Y)$ by
		\[ \widehat{G}(x)= \widehat{F}(x)+ \left [\left (1+\frac{\varepsilon}{2}\right )\widehat{g}(x)-\widehat{F}^*(y_1^*)(x)\right ]y_1\quad \forall \, x \in \mathcal{F}(M), \]
		where $y_1=\sigma(y_1^*)$. Then we have that
		\[ \|\widehat{G} -\widehat{F}\|\leq \frac{\varepsilon}{2}\|\widehat{g}\|+\|\widehat{g}-\widehat{F}^*(y_1^*)\|\leq \frac{\varepsilon}{2}+\gamma<\varepsilon.\]
		Therefore, it is enough to show that $\widehat{G}$ attains its norm at a molecule of $M$. Since for every $y^* \in Y^*$ one has
		\[ 	\widehat{G}^*(y^*)=\widehat{F}^*(y^*)+y^*(y_1)\left (\frac{\varepsilon}{2}\widehat{g}+\widehat{g}-\widehat{F}^*(y^*_1)\right ),\]
		given $y^* \in A\setminus\{y^*_1\}$, we have that
		\[ \|\widehat{G}^*y^*\|\leq 1+\rho(y^*_1)\left(\frac{\varepsilon}{2}+\gamma \right) \leq 1+r\left (\frac{\varepsilon}{2}+\gamma\right ).\]
		On the other hand, for $y^*=y^*_1$ we get that $\widehat{G}^*(y_1^*)=\left (1+\frac{\varepsilon}{2}\right )\widehat{g}$, so
		\[ \|\widehat{G}^*(y_1^*)\|=\left (1+\frac{\varepsilon}{2}\right )\|\widehat{g}\|>\left (1+\frac{\varepsilon}{2}\right )(1-\gamma)>1+r\left (\frac{\varepsilon}{2}+\gamma\right ).	\]
		Consequently, $\|\widehat{G}^*\|=\|\widehat{G}^*(y_1^*)\|$, but $\widehat{G}^*(y_1^*)$ is a multiple of $\widehat{g}$, so it attains its norm as a functional on $\mathcal{F}(M)$ at $u$, hence $\widehat{G}$ attains its norm at the molecule $u \in \Mol(M)$, as desired.
	\end{proof}

Let us show that the above result does not work for the Lip-BPB property. In order to do that, we need the following preliminary result, which is a Lipschitz version of \cite[Proposition 2.3]{acklm}. For the reader's convenience, we include a sketch of the proof based on the one of \cite[Theorem 2.1]{absolutesums} (which is actually stated in a more general form).

\begin{lem}\label{recdomain}
	Let $M$ be a pointed metric space and let $Y$ be a Banach space such that $Y=Y_1\oplus_\infty Z$ for suitable closed subspaces $Y_1$ and $Z$. If the pair $(M,Y)$ has the Lip-BPB property with a function $\eta(\varepsilon)$, then $(M,Y_1)$ also has the Lip-BPB property with the same function.
\end{lem}

\begin{proof}
Fix $\varepsilon >0$, let $\eta(\varepsilon)>0$ be the constant given by the Lip-BPB property of $(M,Y)$ and consider $\widehat{F}_1 \in \operatorname{L}(\mathcal{F}(M),Y_1)$ with $\|F_1\|_L=1$. In the proof of \cite[Theorem 2.1]{absolutesums}, it is shown that if $x_0\in S_{\F(M)}$ satisfies $\|\widehat{F}_1(x_0)\|>1-\eta(\varepsilon)$, then there exist $\widehat{G}_1 \in \operatorname{L}(\mathcal{F}(M),Y_1)$ and $x_1 \in S_{\F(M)}$ such that
	\[ \|\widehat{G}_1(x_1)\|=\|G_1\|_L=1,\quad \|F_1-G_1\|_L< \varepsilon, \quad \|x_0-x_1\|< \varepsilon.\]
	It is enough to take $m \in \Mol(M)$ as $x_0$ and use the Lip-BPB property of $(M,Y)$ instead of the BPB property of $(\F(M),Y)$, so we can ensure that the point $x_1$ is a molecule of $M$.
\end{proof}

The following example, based on \cite[Example 4.1]{acklm}, shows that Proposition \ref{quasi-beta} does not hold for the Lip-BPB property.

\begin{ejem}\label{quasibetano}
	For each $k\in \mathbb{N}$ with $k \geq 2$, consider $Y_k=\mathbb{R}^2$ endowed with the norm
	\[ \|(x,y)\|=\max \left\{|x|,|y|+\frac{1}{k}|x|\right \} \quad \forall \, (x,y) \in \mathbb{R}^2.\]
	Observe that $B_{Y_k}$ is the absolutely convex hull of the set $\{(0,1), (1,1-\frac{1}{k}),  (-1,1-\frac{1}{k})\}$, so each $Y_k$ is polyhedral. Consequently, $Y_k$ has property $\beta$ (see \cite{lindens}). Now, consider the metric space $M=\{0,1,2\}$ with the usual metric. By Corollary \ref{corfinite}, we know that the pair $(M,\mathbb{R})$ has the Lip-BPB property.  Besides, $Y= [\oplus_{k\in\mathbb{N}} Y_k ]_{c_0}$ has property quasi-$\beta$ by \cite[Proposition 4]{aap}, since it is a $c_0$-sum of Banach spaces having that property. However, the pair $(M,Y)$ fails the Lip-BPB property.
\end{ejem}

\begin{proof}
	Fix $0<\varepsilon<\frac{1}{2}$ and assume that there exists $\eta(\varepsilon)>0$ such that $(M,Y_k)$ has the Lip-BPB property with this function for that $\varepsilon$ for every $k \in \mathbb{N}$, $k\geq 2$. That is, for every $k\geq 2$, for every $\widehat{F}_k\in \operatorname{L}(\mathcal{F}(M),Y_k)$ with $\|\widehat{F}_k\|=1$, and every $m_k\in \Mol(M)$ such that $\|\widehat F_k(m_k)\|>1-\eta(\varepsilon)$, there exist $\widehat G_k \in \operatorname{L}(\mathcal{F}(M),Y_k)$ and $u_k \in \Mol(M)$ such that
	\[
	\|\widehat{G}_k(u_k)\|=\|G_k\|_L=1,  \quad \|\widehat F_k-\widehat G_k\| <\varepsilon, \quad \|m_k-u_k\|<\varepsilon. \]
	Recall that $\mathcal{F}(M)$ is two-dimensional and that $m_{0,2}=\frac12 m_{0,1} + \frac12 m_{1,2}$, so $B_{\mathcal{F}(M)}=\overline{\co}\{\pm m_{0,1},\pm m_{1,2}\}$ is a square.
	For every $k\geq 2$, define $\widehat{F}_k \colon \mathcal{F}(M) \longrightarrow Y_k$ by
	\[ \widehat{F}_k(m_{0,1})=\left (-1,1-\frac{1}{k}\right ) \quad \mbox{ and } \quad \widehat{F}_k(m_{1,2})=\left (1,1-\frac{1}{k}\right ) .\]
	Clearly $\|F_k\|_L=1$ and $\widehat{F}_k(m_{0,2})=\widehat{F}_k\left (\frac{1}{2}m_{0,1} + \frac{1}{2}m_{1,2}\right )=\left (0,1-\frac{1}{k}\right )$. Hence, $\|\widehat{F}_k(m_{0,2})\|=1-\frac{1}{k}$.  Then, for every $k \in \mathbb{N}$ such that $1-\frac{1}{k} > 1-\eta(\varepsilon)$, we may find $\widehat{G}_k \colon \mathcal{F}(M) \longrightarrow Y_k$ and $u_k \in \Mol(M)$ such that
	\[ \|\widehat{G}_k(u_k)\|=\|G_k\|_L=1 \quad \|F_k-G_k\|_L<\varepsilon \quad \|u_k-m_{0,2}\|<\varepsilon.  \]
	A straightforward application of Lemma 1.3 in \cite{articulo1} shows that
	\[ \|m_{0,2}-m_{0,1}\|, \|m_{0,2}-m_{1,2}\| \geq 1. \]
	Hence, $u_k=m_{0,2}$ for every $k \in \mathbb{N}$ such that $1-\frac{1}{k}>1-\eta(\varepsilon)$. As $u_k=m_{0,2}=\frac{1}{2}m_{0,1}+\frac{1}{2}m_{1,2}$ and $\|\widehat{G}_k(u_k)\|=1$, it follows that the whole interval $[\widehat{G}_k(m_{0,1}), \widehat{G}_k(m_{1,2})]$ lies on $S_{Y_k}$, so $\widehat{G}_k(m_{0,1})$ and $ \widehat{G}_k(m_{1,2})$ belong to the same face of $B_{Y_k}$. As a consequence, by the shape of $B_{Y_k}$, we obtain that $\|\widehat{G}_k(m_{0,1})-\widehat{G}_k(m_{1,2})\|\leq 1$. Furthermore, since $\|F_k-G_k\|_L<\varepsilon$, we have that
	\[ \|\widehat{F}_k(m_{0,1})-\widehat{G}_k(m_{1,2})\|\leq \|\widehat{F}_k (m_{0,1})-\widehat{G}_k(m_{0,1})\|+\|\widehat{G}_k(m_{0,1})-\widehat{G}_k(m_{1,2})\|<\varepsilon+1<\frac{3}{2}. \]
	On the other hand, since $\|\widehat{F}_k(m_{0,1})-\widehat{F}_k(m_{1,2})\|=2$,
	\[ \|\widehat{F}_k(m_{0,1})-\widehat{F}_k(m_{0,2})\|\geq \|\widehat{F}_k(m_{0,1})-\widehat{F}_k(m_{0,2})\|-\|\widehat{F}_k(m_{1,2})-\widehat{G}_k(m_{1,2})\|>2-\varepsilon>\frac{3}{2}, \]
	which is a contradiction. Note that  $Y=[\oplus_{k\in\mathbb{N}} Y_k]_{c_0}$, so Lemma \ref{recdomain} implies that $(M,Y)$ does not have the Lip-BPB property.
	\end{proof}
	
	From now on in this section, we will focus our attention on Lipschitz compact maps. Let $M$ be a pointed metric space, $Y$ be a Banach space, and $F \colon M \longrightarrow Y$ be a Lipschitz map. We say that $F$ is \textit{Lipschitz compact} when its Lipschitz image, that is, the set
	\[ \left \{\frac{F(p)-F(q)}{d(p,q)} \colon p,q \in M, \, p\neq q \right \}\subseteq Y,\]
	is relatively compact. We denote by $\Lipc(M,Y)$ the space of Lipschitz compact maps from $M$ to $Y$. Some results related to this notion appear in \cite{lipschitzcompact}. Let us make two comments. First, observe that if $Y$ is finite-dimensional, then all Lipschitz maps are Lipschitz compact. Second, it is immediate that a Lipschitz map $F\colon M \longrightarrow Y$ is Lipschitz compact if, and only if, its associated linear operator $\widehat{F}\colon \mathcal{F}(M) \longrightarrow Y$ is compact.
	
	Our aim is to study for which pointed metric spaces $M$ and Banach spaces $Y$ the pair $(M,Y)$ has the Lip-BPB property for Lipschitz compact maps. The first result we present is a slight modification of Theorem \ref{theorem-Gromovunifconcave}.
	
	\begin{prop}
		Let $M$ be a uniformly Gromov rotund pointed metric space. Then, $(M,Y)$ has the Lip-BPB property for Lipschitz compact maps for every Banach space $Y$.
	\end{prop}

	\begin{proof}
		It is enough to note that if we take a Lipschitz compact map $F \in \Lipc(M,Y)$ with $\|F\|_L=1$, then the Lipschitz maps $G_0$ and $H$ which appear in the proof of Theorem \ref{theorem-Gromovunifconcave} will also be Lipschitz compact. This is because $\widehat{G_0}-\widehat{F}$ is a rank-one operator and $\widehat{H}$ is obtained, following the proof of \cite[Proposition 1]{lindens}, as the limit of a sequence of compact operators.
	\end{proof}

	As in Section \ref{sec-universal}, this proposition have two interesting corollaries.
	
	\begin{cor}
		Let $M$ be a H\"older metric space. Then, for every Banach space $Y$ the pair $(M,Y)$ has the Lip-BPB property for Lipschitz compact maps.
	\end{cor}
	
	\begin{cor}
		Let $M$ be a concave pointed metric space such that $\mathcal{F}(M)$ has property $\alpha$. Then, for every Banach space $Y$ the pair $(M,Y)$ has the Lip-BPB property for Lipschitz compact maps.
	\end{cor}

	In this case, it does not make sense to give an analogous result of Corollary \ref{corfinite}, because if $M$ is a finite pointed metric space, then $\Lip(M,Y)=\Lipc(M,Y)$ and so the result is the same.

	In order to get more results, note that it is immediate from its proof that Proposition \ref{propbeta} also holds for Lipschitz compact maps. Consequently, we obtain the following result.
	
	\begin{prop}\label{propbetac}
		Let $M$ be a pointed metric space such that $(M,\mathbb{R})$ has the Lip-BPB property and let $Y$ be a Banach space satisfying property $\beta$. Then, $(M,Y)$ has the Lip-BPB property for Lipschitz compact maps.
	\end{prop}
	
	Furthermore, the hypothesis that $Y$ satisfies property $\beta$ cannot be removed, as the following example shows. It is just a rewriting of Example \ref{finito no}.
	
	\begin{ejem}
		Consider the metric space $M=\{0,1,2\}$ with the usual metric and let $Y$ be a strictly convex Banach space which is not uniformly convex. Then, $(M,\mathbb{R})$ has the Lip-BPB property (for Lipschitz compact maps), but $(M,Y)$ fails the Lip-BPB property for Lipschitz compact maps.
	\end{ejem}
	
	Given a pointed metric space $M$ and a Banach space $Y$, we denote by $\SA_K(M,Y)$ the set of those Lipschitz compact maps from $M$ to $Y$ which strongly attain their norm, that is,
	\[ \SA_K(M,Y)=\SA(M,Y)\cap\Lipc(M,Y).\]

	As before, if we analyze the proof of Proposition \ref{quasi-beta} we can give an analogous result for the density of $\SA(M,Y)$ more general than Proposition \ref{propbetac}. In that proof, we see that $\widehat{G}-\widehat{F}$ is a rank-one operator, so $\widehat{G}$ will be compact if $\widehat{F}$ is. Consequently, we obtain the following result.
	
	\begin{prop}
		Let $M$ be a pointed metric space such that $\SA(M,\mathbb{R})$ is dense in $\Lip(M,\mathbb{R})$ and let $Y$ be a Banach space having property quasi-$\beta$. Then, we have that
		\[ \overline{\SA_K(M,Y)}=\Lipc(M,Y).\]
	\end{prop}

	Furthermore, recall that Example \ref{quasibetano} gave us a finite pointed metric space $M$ and a Banach space $Y$ having property quasi-$\beta$ such that $(M,Y)$ fails the Lip-BPB property. The Lip-BPB property and Lip-BPB property for Lipschitz compact maps are equivalent in this case, so we conclude that the previous result is not true for the Lip-BPB property for Lipschitz compact maps in general.

	The proposition below will be a useful tool in order to carry the Lip-BPB property for compact maps from some sequence spaces to function spaces.
	
	\begin{prop}\label{projections}
		Let $M$ be a pointed metric space and let $Y$ be a Banach space. Suppose that there exists a net of norm-one projections $\{Q_\lambda\}_{\lambda \in \Lambda} \subset \operatorname{L}(Y,Y)$ such that $\{Q_\lambda(y)\}\longrightarrow y$ in norm for every $y \in Y$. If there is a function $\eta\colon \mathbb{R}^+ \longrightarrow \mathbb{R}^+$ such that for every $\lambda \in \Lambda$, the pair $(M,Q_\lambda(Y))$ has the Lip-BPB property for Lipschitz compact maps witnessed by the function $\eta$ , then the pair $(M,Y)$ has the Lip-BPB property for Lipschitz compact maps.
	\end{prop}

	\begin{proof}
		It is enough to repeat the proof of Proposition 2.5 in \cite{dgmm}, taking the point $x_0$ as a molecule of $\mathcal{F}(M)$ and using the Lip-BPB property for Lipschitz compact maps instead of the BPBp for compact operators.
	\end{proof}

	The following result concerning preduals of $L_1$-spaces is based on \cite[Theorem 4.2]{abcck}.

	\begin{prop}\label{L1}
		Let $M$ be a pointed metric space such that $(M,\mathbb{R})$ has the Lip-BPB property. Let $Y$ be a Banach space such that $Y^*$ is isometrically isomorphic to an $L_1$-space. Then, $(M,Y)$ has the Lip-BPB property for Lipschitz compact maps.
	\end{prop}

	\begin{proof}
		Fix $\varepsilon>0$. Consider $\eta(\varepsilon)$ the function given by  the Lip-BPB property of $(M,\mathbb{R})$. Since $\ell_\infty^n$ has property $\beta$ for every $n \in \mathbb{N}$ with constant $\rho=0$, we may apply Proposition \ref{propbeta} to obtain that all the pairs $(M,\ell_\infty^n)$ has the Lip-BPB property (for Lipschitz compact maps) witnessed by the same function $\varepsilon \longmapsto \eta(\varepsilon)$ (see Remark \ref{remark:betaonlyrho}). We take \[\eta'=\min\left \{\frac{\varepsilon}{4}, \eta\left (\frac{\varepsilon}{2}\right )\right \}>0. \]
		Now, consider $F \in \Lipc(M,Y)$ with $\|F\|_L=1$ and $m \in \Mol(M)$ such that $\|\widehat{F}(m)\|>1-\eta'$. Let us take $0<\delta<\frac{1}{4}\min\left \{\frac{\varepsilon}{4}, \|\widehat{F}(m)\|-1+\eta\left( \frac{\varepsilon}{2}\right )\right \}$ and let $\{y_1,\ldots, y_n\}$ be a $\delta$-net of $\widehat{F}(B_{\mathcal{F}(M)})$. In view of \cite[Theorem 3.1]{ll}, we can find a subspace $E \subset Y$ isometric to $\ell_\infty^m$ for some natural $m \in \mathbb{N}$ and such that $d(y_i,E)<\delta$ for every $i\in \{1,\ldots, n\}$. Let $P\colon Y \longrightarrow Y$ be a norm-one projection onto $E$. We will check that $\|P\widehat{F}-\widehat{F}\|<4\delta$. In order to show that, fix $x\in B_{\mathcal{F}(M)}$. Then, there exists $i \in \{1,\ldots,n\}$ such that $\|\widehat{F}(x)-y_1\|<\delta$. Let $e \in E$ be such that $\|e-y_i\|<\delta$. Then, we have that
		\begin{align*}
		\|\widehat{F}(x)-P\widehat{F}(x)\|&\leq \|\widehat{F}(x)-y_i\|+\|y_i-e\|+\|e-P\widehat{F}(x)\| \leq 2\delta + \|P(e)-P\widehat{F}(x)\|\\
		&\leq 2\delta+\|e-\widehat{F}(x)\|\leq 2\delta+ \|e-y_i\|+\|y_i-\widehat{F}(x)\|<4\delta.
		\end{align*}
		So $\|P\widehat{F}\|>\|\widehat{F}\|-4\delta=1-4\delta>0$, which implies that
		\[ \|P\widehat{F}(m)\|>\|\widehat{F}(m)\|-4\delta>1-\eta\left (\frac{\varepsilon}{2}\right ).\]
		Hence, the operator $\widehat{R}=\frac{P\widehat{F}}{\|P\widehat{F}\|}$ verifies that $\|\widehat{R}(m)\|>1-\eta\left (\frac{\varepsilon}{2}\right )$. Since the pair $(M,\ell^m_\infty)$ has the Lip-BPB property for Lipschitz compact maps witnessed by the function $\eta(\varepsilon)$ and $E\subset Y$ is isometrically isomorphic to $\ell_\infty^m$, we can find a Lipschitz compact map $G \in \Lipc(M,E) \subseteq \Lipc(M,Y)$ and $u \in \Mol(M)$ such that
		\[ \|\widehat{G}(u)\|=\|G\|_L=1, \quad \|\widehat{G}-\widehat{R}\|<\frac{\varepsilon}{2}, \quad \|m-u\|<\frac{\varepsilon}{2}.\]
		Finally, we have that
		\[ \|\widehat{G}-\widehat{F}\|\leq \|\widehat{G}-\widehat{R}\|+\|\widehat{R}-P\widehat{F}\|+\|P\widehat{F}-\widehat{F}\|<\frac{\varepsilon}{2}+1-\|P\widehat{F}\|+4\delta<\frac{\varepsilon}{2}+8\delta<\varepsilon.\qedhere\]
	\end{proof}

		We can give a result analogous to Proposition \ref{projections} for the density of $\SA(M,Y)$. The proof of the following result is a slight modification of the proof of that proposition.
	
	\begin{prop}\label{prop_Qlambda_density}
		Let $M$ be a pointed metric space and $Y$ be a Banach space. Suppose that there exists a net of norm-one projections $\{Q_\lambda\}_{\lambda \in \Lambda} \subset \operatorname{L}(Y,Y)$ such that $\{Q_\lambda(y)\}\longrightarrow y$ in norm for every $y \in Y$. If $\SA_K(M,Q_\lambda(Y))$ is dense in $\Lipc(M,Q_\lambda(Y))$ for every $\lambda\in \Lambda$, then
		\[\overline{\SA_K(M,Y)}=\Lipc(M,Y).\]
	\end{prop}

	A first consequence of the above proposition is that for a Lindenstrauss space $Y$ (i.e.\ $Y$ is an isometric predual of an $L_1(\mu)$ space), the density of strongly norm attaining Lipschitz functionals passes to the density of strongly norm attaining compact Lipschitz maps.

	\begin{cor}
	Let $M$ be a pointed metric space such that 	$\SA(M,\mathbb{R})$ is dense in $\Lip(M,\R)$ and let $Y$ be a Banach space such that $Y^*$ is isometrically isomorphic to an $L_1(\mu)$-space. Then, $\SA_K(M,Y)$ is dense in $\Lipc(M,Y)$.
	\end{cor}

	Indeed, it is enough to take into account the already used classical result by Lazar and Lindenstrauss \cite[Theorem 3.1]{ll} that every finite subset of a Lindenstrauss space is ``almost'' contained in a subspace of it which is isometrically isomorphic to an $\ell_\infty^n$ space. Now, all these subspaces are one-complemented and have property $\beta$, so we are in the hypothesis of Proposition \ref{prop_Qlambda_density} by Proposition \ref{quasi-beta}.
		
	Finally, let us present a result concerning the following well known property. A Banach space $X$ is said to have the (Grothendieck) \textit{approximation property} if for every compact set $K$ and every $\varepsilon>0$, there is a finite-rank operator $R \in \operatorname{L}(X,X)$ such that $\|x-R(x)\|<\varepsilon$ for every $x \in X$.
	The following preliminary result, based in \cite[Proposition 4.4]{martin}, is completely elemental.
	
	\begin{prop}\label{approximation}
		Let $M$ be a pointed metric space and let $Y$ be a Banach space with the approximation property. Suppose that for every finite-dimensional subspace $W$ of $Y$, there exists a closed subspace $Z$ such that $W\leq Z\leq Y$ and satisfying that $\overline{\SA_K(M,Z)}=\Lipc(M,Z)$. Then, $\overline{\SA_K(M,Y)}=\Lipc(M,Y)$
	\end{prop}

	If $Y$ is a polyhedral Banach space (i.e.\ for every finite-dimensional subspace its unit ball is the convex hull of finitely many points), then every finite-fimensional subspace of $Y$ has property $\beta$, so Proposition \ref{approximation} and Proposition \ref{quasi-beta} give us the following result.
	
	\begin{cor}
		Let $M$ be a pointed metric space and let $Y$ be a polyhedral Banach space with the approximation property. If $\SA(M,\mathbb{R})$ is dense in $\Lip(M,\mathbb{R})$,  then
		\[\overline{\SA_K(M,Y)}=\Lipc(M,Y).\]
	\end{cor}

Let us comment that Example \ref{quasibetano} shows that the above corollary does not hold for the Lip-BPB property. Indeed, a family $\{Y_k\colon k\geq 2\}$ of two-dimensional polyhedral Banach spaces is constructed there such that considering the metric space $M=\{0,1,2\}$ with the usual metric and writing $Y= [\oplus_{k\in\mathbb{N}} Y_k ]_{c_0}$, then the pair $(M,Y)$ fails the Lip-BPB property. But $Y$ is polyhedral as it is a $c_0$-sum of finite-dimensional polyhedral spaces, and $(M,\R)$ has the Lip-BPBp by Corollary \ref{corfinite}.

\vspace*{0.5cm}

\noindent \textbf{Acknowledgment:\ } The authors are grateful to Luis Carlos Garc\'{\i}a-Lirola and Abraham Rueda Zoca for many comments on the contents of this paper which have been very useful to us. They also thank the anonymous referees for their valuable suggestions, which have helped to improve the exposition of this paper.

\end{document}